\documentclass[pdflatex,sn-mathphys-num]{sn-jnl}

\usepackage{graphicx}%
\usepackage{multirow}%
\usepackage{amsmath,amssymb,amsfonts}%
\usepackage{amsthm}%
\usepackage{mathrsfs}%
\usepackage[title]{appendix}%
\usepackage{xcolor}%
\usepackage{textcomp}%
\usepackage{manyfoot}%
\usepackage{booktabs}%
\usepackage{algorithm}%
\usepackage{algorithmic}%
\usepackage{listings}%
\usepackage{anyfontsize}

\newtheorem{theorem}{Theorem}
\newtheorem{proposition}[theorem]{Proposition}%
\newtheorem{remark}{Remark}%
\newtheorem{lemma}{Lemma}
\newtheorem{assumption}{Assumption}

\newtheorem{corollary}{Corollary}
\newtheorem{definition}{Definition}%

\newcommand{\rev}{\textcolor{black}}

\begin{document}

\title[Probabilistic Iterative Hard Thresholding for Sparse Learning]{Probabilistic Iterative Hard Thresholding for Sparse Learning}


\author[1]{\fnm{Matteo} \sur{Bergamaschi}}\email{bergamas@math.unipd.it}

\author[2]{\fnm{Andrea} \sur{Cristofari}}\email{andrea.cristofari@uniroma2.it}

\author[3]{\fnm{Vyacheslav} \sur{Kungurtsev}}\email{kunguvya@fel.cvut.cz}

\author*[1]{\fnm{Francesco} \sur{Rinaldi}}\email{rinaldi@math.unipd.it}

\affil[1]{\orgdiv{Department of Mathematics ``Tullio Levi-Civita''}, \orgname{University of Padua}}

\affil[2]{\orgdiv{Department of Civil Engineering and Computer Science Engineering}, \orgname{University of Rome ``Tor Vergata''}}

\affil[3]{\orgdiv{Department of Civil Engineering and Computer Science Engineering}, \orgname{Czech Technical University in Prague}}

    
\abstract{
For statistical modeling wherein the data regime is unfavorable in terms of dimensionality relative to the sample size, finding hidden sparsity in the relationship structure between variables can be critical in formulating an accurate statistical model. The so-called ``$\ell_0$ norm'', which counts the number of non-zero components in a vector, is a strong reliable mechanism of enforcing sparsity when incorporated into an optimization problem \rev{for minimizing the fit of a given model to a set of observations}. However, in big data settings wherein noisy estimates of the gradient must be evaluated out of computational necessity, the literature is scant on methods that reliably converge. In this paper, we present an approach towards solving expectation objective optimization problems with cardinality constraints. We prove convergence of the underlying stochastic process and demonstrate the performance on two Machine Learning problems.}




\maketitle

\section{Introduction}

In this paper, we consider the \rev{optimization problem defined as the minimization of an expectation objective subject to a constraint on the cardinality, that is, the number of non-zeros components in the decision vector. Formally,}
\begin{equation}\label{eq:prob}
    \begin{array}{rl}
\min\limits_{x\in\mathbb{R}^n} & f(x):=\mathbb{E}[F(x,\xi)] \\
\text{s.t. } & \|x\|_0\le K,
    \end{array}
\end{equation}
where $f(\cdot)$ is $L(f)$-Lipschitz continuously differentiable\rev{, that is, $\|\nabla f(x) - \nabla f(y)\| \le L(f) \|x-y\|$} for all $x, y \in \mathbb{R}^n$, \rev{with $L(f) > 0$}. We say that $x\in C_K \rev{\subseteq \mathbb{R}^n}$ if $\|x\|_0\le K$ and thus a feasible $x$ corresponds to $x\in C_K$.

This optimization problem is particularly important in \rev{data science applications}. In particular, the expectation objective serves to quantify the minimization of some empirical loss function that enforces the fit of a statistical model to empirical data. Cardinality constraints enforce sparsity in the model, enabling the discovery of the most salient features as far as predicting the label. 

Cardinality constraints present a significant challenge to optimization solvers. The so-called (as it is not, formally)
zero ``norm'' is a discontinuous function that results in a highly non-convex and disconnected feasible set, as well as an unusual topology of stationary points and minimizers~\cite{lammel2022nondegenerate,lammel2023critical}. Algorithmic development 
has been, as similar to many such problems, a parallel endeavor from the mathematical optimization and the
machine learning communities. When dealing with a deterministic objective function, procedures attuned to the structure of the problem and seeking stationary points
of various strength are presented, for instance, in~\cite{beck2013sparsity}. Methods for deterministic optimization problems with sparse symmetric sets are proposed  in, e.g., \cite{beck2016minimization, hallak2024path},  while methods for deterministic optimization problems with 
both cardinality and nonlinear constraints are described in, e.g.,   \cite{Branda2018,bucher2018second,Burdakov2016,vcervinka2016constraint,lapucci2021convergent,lapucci2023unifying,Lu2013}. Simultaneously
works appearing in machine learning conferences, e.g.,~\cite{zhou2018efficient,zhou2019stochastic,murata2018sample,jain2014iterative}, exhibit weak theoretical convergence guarantees,
but appear to scale more adequately as far as numerical experience. Thus, an algorithm that enjoys both reliable performance together with strong theoretical guarantees, as sought for  the high dimensional high data volume model fitting problems in contemporary data science, is as of yet unavailable.

In this paper we attempt to reconcile these two and present an algorithm that is associated with reasonably
strong theoretical convergence guarantees, while at the same time able to solve large scale problems of interest
in statistics and machine learning. To this end, we present a procedure under the framework of \emph{Probabilistic
Models}, which can be understood as a sequential linear Sample Average Approximation (SAA) scheme for solving
problems with statistics in the objective function. First introduced in~\cite{bastin2006adaptive}, then rediscovered with extensive analysis in~\cite{bandeira2014convergence,chen2018stochastic}, this
approach can exhibit asymptotic (and even worst case complexity)
results to a local minimizer of the original problem, while still allowing the use of Newton-type second order iterations of subproblem solutions, and thus faster convergence as far as iteration count. The use of probabilistically accurate estimates within a certain bound in these methods permit a rather flexible approach to estimating the gradient, including techniques that introduce bias, while foregoing the necessity of a stepsize asymptotically diminishing to zero. However, asymptotic accurate convergence still requires increasing the batch size, so the tradeoffs in precision and certainty relative to computation become apparent, and adaptive for the user, in deciding at which point to stop the algorithm and return the current iterate as an estimate of the solution. 

As contemporary Machine Learning applications, we shall consider
Adversarial Attacks (see, e.g., \cite{carlini2017towards,croce2019sparse,modas2019sparsefool} and references therein for further details) and Probabilistic Graphical Model training (see, e.g., \cite{behdin2023sparse,negri2023conditional} and references therein for further details).
In this paper, we shall see how the use of a stochastic gradient and hard sparsity constraint can improve the performance and model quality in the considered problems. 

The paper is organized as follows:
In Section~\ref{sec:background}, we introduce some basic  definitions and  preliminary results related to optimality conditions of problem \eqref{eq:prob} that   ease the theoretical analysis.
We then describe the details of the proposed algorithmic scheme in Section~\ref{sec:algorithm}. We then prove almost sure convergence to suitable stationary points in Section \ref{sec:convergence}.
Numerical results on some relevant Machine Learning applications are reported Section~\ref{sec:numerical}. Finally, we draw some conclusions and discuss some possible extensions in Section~\ref{sec:conclusions}.


\section{Background}\label{sec:background}
Cardinality constrained optimization presents an extensive hierarchy of stationarity conditions, as due to the geometric complexity of the feasible set. This necessitates specialized notions of projection and presents complications due to the projection operation's generic non-uniqueness. 

\paragraph{Definitions and Preliminaries}
\rev{Given a vector $x\in\mathbb{R}^n$, we denote the $i$th component of $x$ by $[x]_i$ and the subvector related to the components with indices in $I\subseteq[1:n]$ by $[x]_I$, while the} 
active and inactive set of $x$ are respectively denoted by
\[
I_{\mathcal{A}}(x):=\{i\in\{1,...,n\},\,\rev{[x]}_i = 0\},\quad 
I_{\mathcal{I}}(x):=\{i\in\{1,...,n\},\,\rev{[x]}_i\neq 0\}.
\]
A set \rev{$T\subseteq\{1,\dots, n\}$ is a \emph{super-support} of $x\in C_K$ if $I_{\mathcal{I}}(x)\subseteq T$ and $\vert T \vert = \rev{K}$.} Let the permutation group of $\{1,...,n\}$ be denoted as $\Sigma_n$ and\rev{,} for a permutation $\sigma \in \Sigma_n$, we write $\rev{[x^{\sigma}]}_i=x_{\sigma(i)}$\rev{, where $\sigma(i)$ denotes the $i$th element of $\sigma$}. For a vector $x\in \mathbb{R}^n$ we denote with $M_i(x)$ the $i$-th largest absolute-value component of $x$, thus we have 
$M_1(x) \rev{\geq} M_2(x)\rev{\geq}\dots\rev{\geq} M_n(x)$. \rev{Let $\sigma^O$ correspond to sorting by this ordering.}

 We finally define the orthogonal projection as
\rev{\[
P_{C_K}(x) \in \arg\min\{\|z-x\|^2,z\in C_K\}=\left\{z\in C_K: \,[z]_{\sigma^O(i)}=M_i(x),\, i\le K,\,[z]_{\sigma^O(i)}=0,\, i> K\right\},
\]}
that is, an $n$-length vector
consisting of the \rev{$K$} components of $x$ with the largest absolute value.
Such operator, as already highlighted in the previous section, is not single-valued due to the inherent non-convexity of the set $C_K$. \rev{For instance consider projecting $(3,1,1)^T$ onto $C_2$. This gives two different points, that is $P_{C_2}((3,1,1)^T)=\{(3,1,0)^T,\ (3,0,1)^T\}$ . A theoretical tool of this kind}
plays a critical role in the development of algorithms for sparsity-constrained optimization (see, e.g., \cite[Section 2]{beck2013sparsity} for a discussion on this matter). 

\paragraph{Optimality Conditions}

Now we define several optimality conditions  for~\eqref{eq:prob}, borrowing heavily from~\cite{beck2013sparsity}. Observe that a notable characteristic of cardinality constrained optimization is the presence of a hierarchy of optimality conditions, that is, a number of conditions that hold at optimal points that range across levels of \rev{restrictiveness or specificity}.

When restricted to a specific support, the ``no descent directions'' rule still provides a necessary optimality condition, which is referred to as basic feasibility. For a full support, \rev{that is for $\|x^*\|_0=K$, this condition aligns with standard stationarity conditions of non-negative directional derivative along any direction in the linearization of the feasible tangent cone. This linearization only includes directions with nonzero components restricted to $I_{\mathcal{I}}(x^*)$. If the support is not full, i.e., $\|x^*\|_0<K$, the stationarity condition must hold for any potential  super support set of $x^*$. The specific definition is given below:}

\begin{definition}\label{def:bf}
    $x^*\in C_K$  is \emph{Basic Feasible} (BF) for problem \eqref{eq:prob} when
    \begin{itemize}
    \item[1.] $\nabla f(x^*)=0$,  if $\|x^*\|_0<K$,
    \item[2.] $\rev{[\nabla f(x^*)]_i=0}$  for all $i\in I_{\cal I}(x^*)$, if $\|x^*\|_0=K$.
    \end{itemize}
\end{definition}
We thus have that when a point $x^*\in C_K$ is optimal for problem \eqref{eq:prob}, then $x^*$ is a BF point (see Theorem 1 in \cite{beck2013sparsity}). The BF property is however a relatively weak necessary condition for optimality, \rev{meaning that a problem could have a potentially large set of suboptimal BF points}. Consequently, stronger necessary conditions are required to achieve higher quality solutions. This is why we use $L$-stationarity, \rev{a stationarity concept with notation similar to minimization over a simple convex set which, in the case of cardinality constraints, considers the ranking of the absolute value of the gradient components}.

\begin{definition}\label{def:ls}
\rev{Given $L>0$, we say that} $x^*\in C_K$ is \emph{$L$-stationary} for problem \eqref{eq:prob}  when
    \begin{equation}\label{eq:lstat}
        x^*\in P_{C_K}\left(x^*-\frac{1}{L}\nabla f(x^*)\right).
    \end{equation}
\end{definition}

An equivalent analytic property of $L$-stationarity is given by the following lemma.
\begin{lemma}\label{lem:lstatb}\cite[Lemma 2.2]{beck2013sparsity}
     $L$-stationarity at $x^*$ is equivalent to $\|x^*\|_0\le K$ and
    \[
    \vert \rev{[\nabla f(x^*)]_i}\vert \begin{cases} \le L M_K(x^*) \quad & i\in I_{\mathcal{A}}(x^*) \\
    =0 \quad & i\in I_{\mathcal{I}}(x^*). \end{cases}
    \]
\end{lemma}

\rev{It can be shown that $L$-stationarity is a more restrictive condition than Basic Feasibility}:
\begin{corollary}
\cite[Corollary 2.1]{beck2013sparsity}
  Suppose that $x^*\in C_k$ is an L-stationary of problem \eqref{eq:prob} for some $L$. Then $x^*$ is BF for problem \eqref{eq:prob}.
\end{corollary}

In addition, the likely intuition that the $L$-stationarity is related to the gradient Lipschitz constant is correct:
\begin{theorem}\cite[Theorem 2.2]{beck2013sparsity}
    If $x^*$ is an optimal solution for problem \eqref{eq:prob} then it is $L$-stationary for all~\mbox{$L>L(f)$}.
\end{theorem}

\rev{To see the distinction between BF and L-stationary, we consider $x^*=(0,1)$ with $\nabla f(x^*)=(-4,0)$ and $K=1$. It is easy to see that $x^*$ satisfies BF but not $L$-stationarity, with $L<4$. This example illustrates how the consideration of the structure of the feasible set for cardinality constrained optimization introduces the necessity of incorporating combinatorial properties associated with ranking gradient components.}

\rev{In this context,} $L$-stationarity is stronger than BF in the sense that the latter only considers a linearized feasible direction stationarity measure, which is also done in the work on sequential conditions in~\cite{kanzow2021sequential}. \rev{Observe that given $x\in \mathbb{R}^n $ and any $d\in\mathbb{R}^n$ such that $d_i\neq 0$ for some 
$i\in I_{\mathcal{A}}(x)$, $x+td\in C_K$ for small $t>0$ is only possible if $\|x\|_0< K$. Thus, if $\|x\|_0=K$, for there to be any such $d_i$, there would have to be some $w\in \mathbb{R}^n$ such that $[w]_j=-[x]_j$ for $j\in I_{\mathcal{I}}(x)$, and the new point would have to be of the form $x+w+t d$. In other words, any feasible  direction from which a zero component becomes non-zero, would first require a swap,  that is the assignment to zero, for another non-zero component, in order to guarantee the cardinality constraint to be satisfied.} Thus there is no feasible  direction in which a zero component becomes \rev{non-zero} when the cardinality constraint is active. $L$-stationarity enables an \rev{exploration of gradient vector components in order to propose directions with distinct support from the current point, which is more aligned to the spirit of cardinality-constrained optimization problems}.



\paragraph{Iterative Hard Thresholding}
An important \rev{ algorithmic tool, often used in the machine learning literature to deal with 
cardinality-constrained problems,} is the \emph{Hard Thresholding Operator} (see, e.g., \cite{beck2013sparsity,blumensath2008iterative} for further details). 
Consider the operator \rev{$\mathbf{HT}(v)$ } applied to a vector $v$ as one that projects
$v$ onto the sparsity constraint, i.e.,
\begin{equation}\label{eq:hardth}
\rev{\mathbf{HT}}(v) \in \arg\min_{w} \left\{\|v-w\|,\,\|w\|_0\le K\right\} := P_{C_K}\left(v\right).
\end{equation}\rev{Iterative Hard Thresholding (IHT) is hence an optimization algorithm designed to solve deterministic sparse optimization problems, particularly those involving cardinality constraints,  which iteratively updates a solution by combining gradient steps with the Hard Thresholding operator:
\begin{equation}\label{eq:IHTscheme}
        x_{k+1}\in P_{C_K}\left(x_k-\alpha\nabla f(x_k)\right), \quad k=0,\ 1,\ 2,\dots.
    \end{equation}
This fixed point method is able to enforce the L-stationary condition in the accumulation points of the generated sequence when a suitable stepsize $\alpha>0$ is chosen (see, e.g., \cite{beck2013sparsity} and references therein for further details).
}


\section{\rev{Probabilistic Iterative Hard Thresholding Algorithm}}\label{sec:algorithm}
\rev{This section introduces the Probabilistic Iterative Hard Thresholding algorithm, a new method that enables to tackle  cardinality-constrained stochastic optimization problems, and the core ideas related to the building blocks of the algorithm, that is the Pseudo Hard Thresholding Operator and the stochastic function estimates. }

\rev{\paragraph{Pseudo Hard Thresholding Operator}
Given $v \in \mathbb{R}^n$, let us define $\tilde{\Sigma}(v)$ as the set of all \emph{sorting permutations}, that is, $\sigma \in \tilde{\Sigma}(v)$ if and only if $[v]_{\sigma(1)} \geq [v]_{\sigma(2)} \geq \ldots \geq [v]_{\sigma(n)}$, where non-uniqueness of the operation arises in case of ties}.
Recall that, the sparse projection operation $P_{C_K}{(v)}$ for a vector $v$ amounts to \rev{sorting $\{\vert [v]_i\vert\}$ in order to define some $\sigma \in \tilde{\Sigma}(|v|)$} and then keeping the $K$ largest magnitude components of $v$ while setting the rest to zero. 

\rev{As observed above, an algorithmic iterative descent procedure would involve the negative of the gradient of $f$, or an estimate thereof}. Indeed, as the objective function is an expectation, we do not have access to the exact value of the $\nabla f(x)$ and hence the magnitude ranking of its components. Thus, \rev{it is necessary for us} to use noisy gradient estimates $\nabla F(x,\hat{\xi})$ to try to \rev{guess the} actual ranking of the component magnitudes.

Asymptotically, we want to ensure that this sparse projector estimates the true ranking at any limit point. Given that the sequence of iterates \rev{provides} incrementally and asymptotically increasing sample sizes  \rev{of points in the neighborhood of the solution limit point}, this presents a natural opportunity to use the algorithm iterate sequence itself to perform this estimate. By relying on consistency in the asymptotic sampling regime, we obtain statistical guarantees on accurate identification. 

Let $x_k$ correspond to the current iterate \rev{of the given algorithm. Now we define our particular sequential ranking estimate for the magnitude of the vector components related to the gradient of $f(x_k)$. }

\rev{During the earlier iterations of the algorithm, there accumulated a set of permutations $S_k=\{\sigma^{(j)}\}_{j\in[J]},\, \sigma^{(j)}\in \Sigma_n$ with coefficient weights $\{\omega^{(j)}\}_{j\in[J]},\, \omega\in \Delta_J$, where we use $\Delta_m$ to denote the unit simplex of dimension $m$, and $[J]=\{1,\dots, J\}$, with $J=|S_k|$. Now we compute a new gradient estimate at $k$, and perform the inductive step on the procedure of updating the set $S_k$.}

Specifically, \rev{given a noisy evaluation $g_k\approx \nabla f(x_k)$, we take a \emph{clipped} gradient step, wherein we step in the negative direction of the scaled negative gradient \rev{estimate} $-\alpha\min\left\{1,\frac{\delta_k}{\alpha\left\|g_k\right\|}\right\}g_k$, with $\delta_k>0$ the clipping level, which is updated during the iterations (see below), and $\alpha>0$ a constant which, roughly speaking, represents a sort of stepsize along the chosen direction.
It is easy to see that the clipped gradient step  projects the term $-\alpha g_k$ within a ball of  radius $\delta_k$.
Then we consider a sorting permutation of the vector:}
\begin{equation}\label{eq:sigk}
\sigma_k\in\tilde{\Sigma}\left(x_k-\alpha\min\left\{1,\frac{\delta_k}{\alpha\left\|g_k\right\|}\right\}g_k\right),    
\end{equation}


Then, taking 
\begin{equation}\label{eq:ikdef}
\rev{I_k=\{{\sigma}_k(i), \, i=1,\dots, K\}},
\end{equation}
that is, the set of indices whose components are largest \rev{according to ${\sigma}_k$, we introduce} the \textbf{Pseudo Hard Thresholding} operator corresponding to iteration $k$, defined as follows:
\begin{equation}\label{eq:pshtk}
\rev{\mathbf{PHT}^k(v) \in \arg\min_{w} \left\{\|v-w\|,\,[w]_{[1:n]\setminus I_k} = 0\right\}=P_{I_k}(v),}
\end{equation}
\rev{where $P_{I_k}(v)$ is the projection of $v$ over the subspace of $\mathbb{R}^n$ defined by those points having support in~$I_k$.}
\rev{As highlighted above, we take a \emph{clipped} gradient step and then apply our Pseudo-Hard Thresholding operator on this point:
\begin{equation}\label{xk_proj}
\hat x_k = \mathbf{PHT}^k\left(x_k - \alpha\min\biggl\{1,\frac{\delta_k}{\alpha\|g_k\|}\biggr\} g_k\right)= P_{I_k} \biggl(x_k - \alpha\min\biggl\{1,\frac{\delta_k}{\alpha\|g_k\|}\biggr\} g_k\biggr).
\end{equation}}
\rev{Observe that unlike Hard Thresholding, the Pseudo Hard Thresholding operator can be computed in a straightforward closed form expression of simply setting the components not corresponding to the estimated top $K$ to be zero, that is,}
\begin{equation}\label{x_hataltin}
[\hat x_k]_i =
\begin{cases} 
0 & i\notin I_k, \\
\left[x_k-\alpha\min\left\{1,\frac{\delta_k}{\alpha\left\|g_k\right\|}\right\}g_k\right]_i & i\in I_k.
\end{cases}
\end{equation}
\rev{Note that, using known properties of the projection operator~\cite{bertsekas:1999}, for all $v \in \mathbb R^n$ we can write
$(x_k-v-P_{I_k}(x_k-v))^T(x_k-P_{I_k}(x_k-v)) \le 0$. So, recalling~\eqref{xk_proj}, we have
\[
g_k^T(\hat{x}_k-x_k) \le -\frac 1{\alpha} \max\biggl\{1,\frac{\alpha\|g_k\|}{\delta_k}\biggr\} \|\hat{x}_k-x_k\|^2.
\]}

This presents an opportunity to \rev{establish a descent lemma adapted to} the context of cardinality constrained optimization problems. To this end, define
\begin{equation}\label{eq:htil}
h_k(y) = f(x_k) + g_k^T (y-x_k),
\end{equation}
so that
\begin{equation}\label{eq:desc}
h_k(\hat{x}_k)-h_k(x_k) = g_k^T(\hat{x}_k-x_k) \le -\frac 1{\alpha} \max\biggl\{1,\frac{\alpha\|g_k\|}{\delta_k}\biggr\} \|\hat{x}_k-x_k\|^2  \le -\frac 1{\alpha}\|\hat{x}_k-x_k\|^2.
\end{equation}

\bigskip\par


\paragraph{Accuracy Estimates}
\rev{In our algorithm we require that both the function values $f(x_k)$ and $f(x_k+s_k)$ as well as the first-order models used to compute the step are sufficiently accurate with high probability. Here we present some standard definitions that make these notions precise. In particular, the following definitions of accurate function estimates and of accurate models are similar to the ones in \cite{chen2018stochastic}.
}
\begin{definition}\label{def:facc}
Define $s_k=\hat{x}_k-x_k$. The function estimates $f_k^0$ and $f_k^s$ are $\varepsilon_f$-accurate estimates of $f(x_k)$ and $f(x_k+s_k)$, respectively, for a given $\delta_k$ if
\begin{equation}\label{def:accestf}
|f_k^0-f(x_k)|\leq \varepsilon_f\delta_k^2\quad \mbox{and} \quad |f_k^s-f(x_k+s_k)|\leq \varepsilon_f \delta_k^2.
\end{equation}
\end{definition}

\begin{definition}\label{def:modelacc}
The model for generating the iterate is $\kappa$-$\delta_k$, or $(\kappa_f,\kappa_g)$-$\delta_k$ accurate when
\begin{equation}\label{eq:acc}
\|\nabla \rev{f}(y)-g_k\|\le \kappa_g \delta_k \quad \text{and} \quad \big\lvert f(y)-f(x_k)-g_k^T(y-x_k)\big \rvert \le \kappa_f \|y-x_k\|\delta^2_k
\end{equation}
for all \rev{$y$ such that $[y]_{I_k}\in B([x_k]_{I_k},\delta_k)$}.
\end{definition}
    \rev{In the following proposition we report a result implied by equation \eqref{eq:acc} that will be useful for the convergence analysis.
 \begin{proposition}\label{prop:accik}
If~\eqref{eq:acc} holds for all $y$ such that $[y]_{I_k}\in B([x_k]_{I_k},\delta_k)$, then
\begin{equation}\label{eq:accik}
\|
[\nabla f(y)-g_k]_{I_k}\|\le \kappa_g \delta_k \quad  \text{and} \quad  \big\lvert f(y)-f(x_k)-[g_k]_{I_k}^T[y-x_k]_{I_k}\big \rvert \le \kappa_f \|[y-x_k]_{I_k}\|\delta^2_k 
\end{equation}
for all $y$ such that $[y]_{I_k}\in B([x_k]_{I_k},\delta_k)$.
\end{proposition}
\begin{proof}
Let~\eqref{eq:acc} hold.
It is straightforward to verify that the first inequality in~\eqref{eq:accik} follows from the first inequality in~\eqref{eq:acc} since $\|[\nabla f(y)-g_k]_{I_k}\| \le \|\nabla f(y)-g_k\|$. To show that the second inequality in~\eqref{eq:accik} holds as well, assume by contradiction that there exists $y \in \mathbb R^n$ such that $[y]_{I_k}\in B([x_k]_{I_k},\delta_k)$ and
\[
\big\lvert f(y)-f(x_k)-[g_k]_{I_k}^T[y-x_k]_{I_k}\big \rvert > \kappa_f \|[y-x_k]_{I_k}\|\delta^2_k.
\]
Now, define $\tilde y \in \mathbb R^n$ as follows:
\[
[\tilde y]_i =
\begin{cases}
[y]_i \quad & \text{if } i \in I_k, \\
[x_k]_i \quad & \text{if } i \notin I_k.
\end{cases}
\]
Then, $[\tilde y]_{I_k} \in B([x_k]_{I_k},\delta_k)$ and
\[
\big\lvert f(y)-f(x_k)-g_k^T(\tilde y-x_k)\big \rvert =
\big\lvert f(y)-f(x_k)-[g_k]_{I_k}^T[\tilde y-x_k]_{I_k}\big \rvert > \kappa_f \|[y-x_k]_{I_k}\|\delta^2_k = \kappa_f \|y-x_k\|\delta^2_k,
\]
thus contradicting the second inequality in~\eqref{eq:acc}.
\end{proof}
 }

\rev{
\paragraph{Algorithm Description}
See Algorithm \ref{alg:iht} for a summary of the scheme we present in this paper.
At the initialization, a feasible starting point $x_0\in C_K$ and a step $\delta_0$ are chosen. Then, at each iteration, a random variable $\xi_k$ is sampled from a distribution $\Xi$  and the gradient estimate $g_k=\nabla F(x_k, \xi_k)$ at the point $x_k$ with respect to $x$ with the realization defined by sample (e.g. minibatch) $\xi_k$ is calculated (see Step 3 of the algorithm). In Step~4, the sorting permutation $\sigma_k$ 
with largest weight is selected and used to define the set of indices $I_k$ related to the largest components. }


\rev{The Pseudo Hard Thresholding operator is then used to calculate the trial point $\hat x_k$.
The algorithm hence generates an estimate of the true objective function $f$ at the trial point $\hat x_k$. By also computing an estimate of $f$ at the current point $x_k$ in Step 7, it can perform a test, in Step 8, on whether there is a sufficient reduction of the model. If so, we have a \emph{successful iteration}, the iterate is updated to equal the computed estimate $\hat{x}_k$. Moreover, the parameter $\delta_k$ used to define the required accuracy conditions on the gradient estimate as well as the clipping level associated with the step normalization is increased, as long as it is smaller than some large threshold $\delta_{max}$. This relaxation of the subproblem requirements in stringency of accuracy and descent permits for potentially larger steps at the next iteration, promoting faster convergence when estimates are an accurate representation of the underlying objective function (see Step 9).}

\rev{Otherwise, the sample estimate for sufficient decrease is not satisfied, we have an \emph{unsuccessful iteration}, the  $\delta_k$ is reduced, and the iterate stays the same (see Step 11). However, it is important to note that a successful iteration does not necessarily give a reduction of the true function $f$. In fact, such a function is not available and the acceptance condition is based only on estimates of $f(x_k)$ and $f(\hat x_k)$. 
}

\begin{algorithm}
	\caption{Probabilistic Iterative Hard Thresholding}
	\label{alg:iht}
	\begin{algorithmic}[1]
		\STATE \textbf{Initialization:} $x_0\in C_K$, $\delta_0 \in (0,\delta_{max}]$, Parameters $\delta_{max} > 0$, $\gamma \in (0,1)$ 
        \FOR{$k=0,1,2,\ldots$}
        \STATE Sample a minibatch $\xi_k\sim \Xi$ and compute $g_k= \nabla F(x_k,\xi_k)$
        \STATE Compute \rev{$\sigma_k$} by~\eqref{eq:sigk} 
		\STATE Compute $\hat{x}_k$ from the \rev{Pseudo Hard Thresholding}~\eqref{x_hataltin}
        \STATE Compute stochastic estimates $f^s_k\approx f(\hat{x}_k),\,f^0_k\approx f(x_k)$
        \IF{$\dfrac{f^0_k-f^s_k}{\|[g_k]_{I_k}\|\delta_k}\ge \eta_1$ and $\|[g_k]_{I_k}\|\ge \eta_2 \delta_k$}
        \STATE Set $\delta_{k+1} = \min \{\gamma \delta_k, \delta_{max}\}$, let $x_{k+1}=\hat{x}_k$ 
        \ELSE
        \STATE Set $\delta_{k+1} = \gamma^{-1} \delta_k$, let $x_{k+1}=x_k$
        \ENDIF
        \ENDFOR 
	\end{algorithmic}
\end{algorithm}

\section{Convergence Theory}\label{sec:convergence}
Now we develop our \rev{analysis} for justifying the long term convergence of \rev{Algorithm \ref{alg:iht}} based on classic arguments on probabilistic models given in~\cite{chen2018stochastic}(see also~\cite{bandeira2014convergence,cartis2018global}). To this end, we remark that the iterates, being dependent on random function and gradient estimates, define a stochastic process $X_k$. The algorithm itself is a realization, thus denoting $x_k=X_k(\omega)$, $\delta_k=\Delta_k(\omega)$, etc. for $\omega$ the random element defining the realization. Similar as to the \rev{previous works}, we consider a filtration with the sigma algebra $\mathcal{F}_k$ defining the start of the iteration and $\mathcal{F}_{k+\frac{1}{2}}$ defining the algebra after the minibatch has been sampled and $g_k$ computed. This filtration will be implicit in the statements of the convergence results.

We begin with a standard assumption \rev{that gives probability bounds on  the accuracy of the linear model defined by $\{G_k\}$\rev{, with $G_k(\omega)=\nabla F(X_k,\omega))$} and the estimates $\{F^0_k, F^s_k\}$ \rev{defined similarly}. To this end define $\theta,\beta$ to be the probabilities that $\{G_k\}$ is  $\kappa$-$\delta_k$-accurate, and $\{F^0_k, F^s_k\}$ are $\varepsilon_f$-accurate, respectively}.
\begin{assumption}\label{as:probacc}
    Given $\theta,\beta\in(0,1)$ and \rev{$\varepsilon_f>0$}, there exist $\kappa_g,\kappa_f$ such that the sequence of  \rev{ gradient estimates $\{G_k\}$ and function estimates $\{F^0_k, F^s_k\}$ generated by Algorithm \ref{alg:iht}   are, respectively, 
    $\kappa$-$\delta_k$-accurate with probability $\theta$ as per Definition~\ref{def:modelacc}, and $\varepsilon_f$-accurate with probability $\beta$ as per Definition~\ref{def:facc}.} 
\end{assumption}

One can expect that \rev{when the estimates are sufficiently close to their deterministic counterparts, the classical sparse optimization theory, namely} \cite[Lemma 3.1]{beck2013sparsity} provides for the guarantee of function decrease. Indeed, one can derive the following lemma which also functionally corresponds to 
\cite[Lemma 4.5]{chen2018stochastic}.

\begin{lemma}\label{lem:gooddec0}
If the model for generating the iterate $k$ is $\kappa$-$\delta_k$ accurate according to Definition~\ref{def:modelacc},
with $\hat x_k$ and $\delta_k$ being such that 
\begin{equation}\label{cond_delta_a}
\delta_k \le \frac1{2\alpha \kappa_{\rev{f}}\delta_{max}} \|x_k-\hat x_k\|,
\end{equation}
then 
\begin{equation}\label{eq:gooddec0}
f(x_k)-f(\hat{x}_{k}) \ge \frac1{2\alpha} \|\hat x_k-x_k\|^2.
\end{equation}
\end{lemma}

\begin{proof}
Using the definition of $h_k$ given in~\eqref{eq:htil}, we can write
\begin{equation*}
\begin{split}
f(\hat{x}_k)-f(x_k) & = f(\hat{x}_k)-h_k(\hat{x}_k)+h_k(\hat{x}_k)-h_k(x_k)+h_k(x_k)-f(x_k) \\
& = f(\hat{x}_k)-h_k(\hat{x}_k)+h_k(\hat{x}_k)-h_k(x_k) \\
& = f(\hat{x}_k)-f(x_k)-g_k^T(\hat{x}_k-x_k) + g_k^T(\hat{x}_k-x_k) \\
& \le \kappa_{\rev{f}} \|x_k-\hat{x}_{k}\|\delta_k^2+g_k^T(\hat{x}_k-x_k),
\end{split}
\end{equation*}
where the inequality follows from the second condition in~\eqref{eq:acc}.
Using~\eqref{eq:desc}, we also have that
$$
g_k^T(\hat{x}_k-x_k) \le - \frac 1{\alpha}\|\hat{x}_k-x_k\|^2.
$$
Then, we obtain
\begin{equation}\label{gooddec_chain}
f(\hat{x}_k)-f(x_k) \le \kappa_{\rev{f}} \|x_k-\hat{x}_{k}\|\delta_k^2 - \frac 1{\alpha}\|\hat{x}_k-x_k\|^2 \le \kappa_{\rev{f}}\delta_{max} \|x_k-\hat{x}_{k}\|\delta_k - \frac 1{\alpha}\|\hat{x}_k-x_k\|^2,
\end{equation}
where the last inequality follows from the fact that $\delta_k \le \delta_{max}$.
Moreover, \eqref{cond_delta_a} implies that
\[
\kappa_{\rev{f}}\delta_{max} \|x_k-\hat{x}_{k}\|\delta_k \le \frac1{2\alpha} \|\hat x_k-x_k\|^2.
\]
Using this inequality in~\eqref{gooddec_chain},
the desired result follows.
\end{proof}

Now, taking inspiration from \cite[Lemma 4.6]{chen2018stochastic}, we can bound the decrease with respect to the projected real gradient.

\begin{lemma}\label{lem:gooddec2}
If the model for generating the iterate $k$ is $\kappa$-$\delta_k$ accurate according to Definition~\ref{def:modelacc} and
\begin{equation}\label{cond_delta_b}
\delta_k \le a\Biggl\|x_k-P_{I_k}\biggl(x_k-\alpha\min\left\{1,\frac{\delta_k}{\alpha\|\nabla f(x_k)\|}\right\}\nabla f(x_k)\biggr)\Biggr\|,
\end{equation}
where
\begin{equation}\label{param_a}
a = \frac 1{2\alpha \kappa_{\rev{f}}\delta_{max}+2\sqrt{K}}
\end{equation}
and
\begin{equation}\label{param_alpha}
\alpha > \dfrac{\sqrt K}{\kappa_{\rev{f}}\delta_{max}},
\end{equation}
then 
\begin{equation}\label{eq:gooddec2}
f(x_k)-f(\hat{x}_{k}) \ge c \Biggl\|x_k-P_{I_k}\biggl(x_k-\alpha\min\left\{1,\frac{\delta_k}{\alpha\|\nabla f(x_k)\|}\right\}\nabla f(x_k)\biggr)\Biggr\|^2,
\end{equation}
with
\[
c = \frac{1-4a\sqrt{K}}{2\alpha} > 0.
\]
\end{lemma}

\begin{proof}
We can write
\begin{equation*}
\begin{split}
& \Biggl\|x_k-P_{I_k}\biggl(x_k-\alpha\min\left\{1,\frac{\delta_k}{\alpha\|\nabla f(x_k)\|}\right\}\nabla f(x_k)\biggr)\Biggr\| \le \\
& \quad \|x_k - \hat x_k\| + \Biggl\|\hat x_k-P_{I_k}\biggl(x_k-\alpha\min\left\{1,\frac{\delta_k}{\alpha\|\nabla f(x_k)\|}\right\}\nabla f(x_k)\biggr)\Biggr\|.
\end{split}
\end{equation*}
Using~\eqref{xk_proj}, we get
\begin{equation}\label{x_hatx}
\begin{split}
& \Biggl\|x_k-P_{I_k}\biggl(x_k-\alpha\min\left\{1,\frac{\delta_k}{\alpha\|\nabla f(x_k)\|}\right\}\nabla f(x_k)\biggr)\Biggr\| = \\
& \quad \|x_k - \hat x_k\| +
\Biggl\|\alpha\min\left\{1,\frac{\delta_k}{\alpha\|g_k\|}\right\}[g_k]_{I_k}-\alpha\min\left\{1,\frac{\delta_k}{\alpha\|\nabla f(x_k)\|}\right\}[\nabla f(x_k)]_{I_k}\Biggr\| = \\
& \quad \|x_k - \hat x_k\| + \delta_k \Biggl\|\min\left\{\frac{\alpha\|g_k\|}{\delta_k}, 1\right\} \frac{[g_k]_{I_k}}{\|g_k\|} -\min\left\{\frac{\alpha\|\nabla f(x_k)\|}{\delta_k}, 1\right\} \frac{[\nabla f(x_k)]_{I_k}}{\|\nabla f(x_k)\|}\Biggr\| \le \\
& \quad \|x_k - \hat x_k\| + 2\sqrt{K}\delta_k,
\end{split}
\end{equation}
where the last inequality follows from the fact that $\|u-v\| \le \sqrt{K}\|u-v\|_\infty\leq 2\sqrt{K}$ for all $u,v \in \mathbb R^K$ such that \rev{$\|u\| \leq 1$ and $\|v\| \leq 1$}.
From~\eqref{cond_delta_b}, the first term in~\eqref{x_hatx} is greater \rev{or} equal \rev{than} $\delta_k/a$, leading to
\[
\frac{\delta_k}a \le \|x_k - \hat x_k\| + 2\sqrt{K}\delta_k.
\]
Using the definition of $a$ given in~\eqref{param_a}, it follows that~\eqref{cond_delta_a} is satisfied and we can apply Lemma~\ref{lem:gooddec0}, obtaining
\begin{equation}\label{f_decr}
f(x_k)-f(\hat{x}_{k}) \ge \frac1{2\alpha} \|\hat x_k-x_k\|^2.
\end{equation}
Finally, in order to lower bound the right-hand side term in the above inequality, using~\eqref{x_hatx} we can write
\begin{equation*}
\begin{split}
\|x_k-\hat{x}_{k}\|^2 \ge &\Biggl(\Biggl \|x_k-P_{I_k}\biggl(x_k-\alpha\min\left\{1,\frac{\delta_k}{\alpha\|\nabla f(x_k)\|}\right\}\nabla f(x_k)\biggr)\Biggr\|- 2\sqrt{K}\delta_k\Biggr)^2 \\
& \ge \Biggl\|x_k-P_{I_k}\biggl(x_k-\alpha\min\left\{1,\frac{\delta_k}{\alpha\|\nabla f(x_k)\|}\right\}\nabla f(x_k)\biggr)\Biggr\|^2 + \\
&\quad -4\sqrt{K}\delta_k\Biggl\|x_k-P_{I_k}\biggl(x_k-\alpha\min\left\{1,\frac{\delta_k}{\alpha\|\nabla f(x_k)\|}\right\}\nabla f(x_k)\biggr)\Biggr\| \\
& \ge \left(1 - 4a\sqrt{K}\right) \Biggl\|x_k-P_{I_k}\biggl(x_k-\alpha\min\left\{1,\frac{\delta_k}{\alpha\|\nabla f(x_k)\|}\right\}\nabla f(x_k)\biggr)\Biggr\|^2,
\end{split}
\end{equation*}
where the last inequality follows from~\eqref{cond_delta_b}.
From~\eqref{param_alpha}, it also follows that $c>0$,
thus leading to the desired result.
\end{proof}

The next lemma states conditions on $\delta_k$ to guarantee that an iteration is successful, similarly as in
\cite[Lemma 4.7]{chen2018stochastic}.

\begin{lemma}\label{lem:stepaccept}
If, at iteration $k$, the estimates $f^0_k,f^s_k$ are $\varepsilon_f$-accurate according to Definition~\ref{def:facc} and the model is $\kappa$-$\delta_k$ accurate according to Definition~\ref{def:modelacc}, with
\[
\delta_k \le \min\Biggl\{\frac 1{\eta_2}, \frac{1-\eta_1}{2 \varepsilon_f+\kappa_{\rev{f}} \delta_{max}}\Biggr\}\|[g_k]_{I_k}\|,
\]
then the step is accepted.
\end{lemma}

\begin{proof}
Define
\[
\rho_k = \frac{f^0_k-f^s_k}{\|[g_k]_{I_k}\|\delta_k}.
\]
Using~\eqref{def:accestf} and~\rev{\eqref{eq:accik}}, we can write 
\rev{
\begin{equation*}
\begin{split}
\rho_k & = \frac{f^0_k-f(x_k)}{\|[g_k]_{I_k}\|\delta_k} + \frac{f(x_k)-f(\hat x_k)}{\|[g_k]_{I_k}\|\delta_k} + \frac{f(\hat x_k)-f^s_k}{\|[g_k]_{I_k}\|\delta_k} \\
& \le \frac{2 \varepsilon_f \delta_k}{\|[g_k]_{I_k}\|} + \frac{\vert [g_k]_{I_k}^T [\hat x_k-x_k]_{I_k}\vert  + \kappa_{\rev{f}} \|[\hat x_k-x_k]_{I_k}\|\delta_k^2}{\|[g_k]_{I_k}\|\delta_k} \\
& \le \frac{2 \varepsilon_f \delta_k}{\|[g_k]_{I_k}\|} + 1 + \frac{\kappa_{\rev{f}} \delta_{max}\delta_k}{\|[g_k]_{I_k}\|},
\end{split}
\end{equation*}}
where the last inequality follows from the fact that \rev{$\vert [g_k]_{I_k}^T [\hat x_k-x_k]_{I_k}\vert \le \|[g_k]_{I_k}\|\|x_k-x_k]_{I_k}\|$ and $\|[\hat x_k-x_k]_{I_k}\| \le \delta_k \le \delta_{max}$}.
Then
\[
|\rho_k-1| \le \frac{(2 \varepsilon_f+\kappa_{\rev{f}} \delta_{max}) \delta_k}{\|[g_k]_{I_k}\|} \le 1 - \eta_1,
\]
where we have used the assumption on $\delta_k$ in the last inequality. Hence, $\rho_k \ge \eta_1$. Since we have also assumed that $\|[g_k]_{I_k}\| \ge \eta_2 \delta_k$, from the instructions of the algorithm (see line~8 of Algorithm~\ref{alg:iht}) it follows that the step is accepted.
\end{proof}

\begin{lemma}\label{lem:stepaccept}
If the estimates $f^0_k,f^s_k$ at iteration $k$ are $\varepsilon_f$-accurate according to Definition~\ref{def:facc}
with $\epsilon_f < (\eta_1 \eta_2)/2$ and the step is accepted, then
\[
f(x_{k+1})-f(x_k) \le - C\|\delta_k\|^2,
\]
with $C = \eta_1 \eta_2 - 2\epsilon_f > 0$.
\end{lemma}

\begin{proof}
Since the step is accepted, from the instructions of the algorithm (see line~8 of Algorithm~\ref{alg:iht}) we can write
\begin{equation}\label{f_decr_step_accepted}
f_k^0 - f_k^s  \ge \eta_1 \|[g_k]_{I_k}\| \delta_k \ge \eta_1 \eta_2 \delta_k^2.
\end{equation}
Moreover,
\[
f(x_k+s_k) - f(x_k) = f(x_k+s_k) - f_k^s + f_k^s - f_k^0 + f_k^0 - f(x_k) \le 2 \epsilon_f \delta_k^2 - \eta_1 \eta_2 \delta_k^2,
\]
where the inequality follows from~\eqref{def:accestf} and~\eqref{f_decr_step_accepted}.
Then, using the definition of $C$ given in the assertion, the desired result follows.
\end{proof}
Now we define the stochastic process
\begin{equation}\label{eq:stochproc}
    \Phi_k := \nu f(x_k)+(1-\nu)\delta_k^2. 
\end{equation}

The next \rev{theorem} is along the lines of Theorem 4.11 in \cite{chen2018stochastic}. The result requires a compactness assumption, which we present first.
\begin{assumption}\label{as:cont}
    
    \rev{Given $\delta_{\max}$ and some initial guess $x_0$, let $\mathcal{L}(x_0,\delta_{\max})$ be the set containing all the iterates generated by the algorithm, noting that this depends on the stochastic realization of the iterates and gradient estimates. Furthermore, let 
    \[
    \bar{\mathcal{L}}(x_0,\delta_{\max}):=\bigcup\limits_{x\in \mathcal{L}(x_0,\delta_{\max})} B(x,\delta_{\max})
    \]
    be the union of $\delta_{\max}$ radius balls around all of the iterates.}
    
    \rev{Assume that, for all realizations considered in the theoretical analysis, the following holds:
    \begin{itemize}
        \item $f$ is bounded on $\mathcal{L}(x_0,\delta_{\max})$,
        \item $f$ and $\nabla f$ are both $L$-Lipschitz continuous on $\bar{\mathcal{L}}(x_0,\delta_{\max})$.
    \end{itemize}}
\end{assumption}

\begin{theorem}
\label{thm:4.11_equivalent}
Let $\{x_k\}$ be the sequence of iterates generated by the Probabilistic Iterative Hard Thresholding Algorithm (Algorithm~\ref{alg:iht}) under Assumption~\ref{as:probacc}, and moreover assume that the function and iterates are such that Assumption~\ref{as:cont} holds. 
Also assume that the step acceptance parameter $\eta_2$ satisfies
\begin{equation}\label{eq:thmetaacc}
\eta_2 \ge 3\kappa_f\alpha  
\end{equation}
and the function accuracy parameter $\varepsilon_f$ is chosen such that,
\begin{equation}\label{eq:thmvarf}
\varepsilon_f \le  \min\left\{\kappa_f,\eta_1\eta_2\right\}.
\end{equation}

Then, \rev{if $\theta$ and $\beta$ are sufficiently large,} it holds that the sequence of trust region bounds $\{\delta_k\}$
satisfies the summability condition
\begin{equation} \label{eq:conclusion_summability}
\sum_{k=0}^{\infty} \delta_k^2 < \infty
\end{equation}
almost surely.
\end{theorem}

\begin{proof}
We define the constants $\zeta$ together with $\nu$ appearing in~\eqref{eq:stochproc} as satisfying
\begin{equation}\label{eq:zetadef}
\zeta \ge \max\left\{ a^{-1} ,\kappa_g +\max\left\{\eta_2,\frac{2\epsilon_f+\kappa_f\delta_{max}}{1-\eta_1}\right\}\right\},
\end{equation}
where we recall that
\[
a = \frac 1{2\alpha \kappa_{\rev{ f}}\delta_{max}+2\sqrt{K}}
\]
and
\begin{equation}\label{eq:nudef}
\frac{\nu}{1-\nu} > \max\left\{\frac{4\gamma^2}{\zeta c},\frac{4\gamma^2}{\eta_1\eta_2},\frac{\gamma^2}{\kappa_f}\right\},
\end{equation}
with $c$ defined by Lemma~\ref{lem:gooddec2}.

We observe that on successful, or accepted, iterations,
\begin{equation}\label{eq:stochprocacc}
    \Phi_{k+1} - \Phi_k \le \nu (f(x_{k+1})-f(x_k))+(1-\nu)(\gamma^2-1)\delta_k^2
\end{equation}
and on unsuccessful iterations,
\begin{equation}\label{eq:stochprocnotacc}
    \Phi_{k+1} - \Phi_k \le(1-\nu)\left(\frac{1}{\gamma^2}-1\right)\delta_k^2 < 0.
\end{equation}

Let us define the event sequence \rev{\(A_k\)} as the satisfaction of model accuracy according to Definition~\ref{def:modelacc}, that is for all $y\in B(x_k,\delta_k)$, \rev{the random event occurs in $\mathcal{F}_{k}$ such that  the realization of of $G_k$ satisfies
\[
\|\nabla \rev{f}(y) - g_k\| \leq \kappa_{\rev{g}} \delta_k \quad \text{and} \quad |f(y) - f(x_k) - g_k^T(y - x_k)| \leq \kappa_{\rev{f}} \|y - x_k\| \delta_k^2.
\]}

\rev{Furthermore, the sequence of events \(B_k\)} is defined as the \rev{random event within $\mathcal{F}_{k+\frac{1}{2}}$ indicating that the function evaluation samples satisfy $\varepsilon_f$-accuracy according to Definition~\ref{def:facc}, that is,
\[
|f_k^0 - f(x_k)| \leq \varepsilon_f \delta_k^2 \quad \text{and} \quad |f_k^s - f(x_k + s_k)| \leq \varepsilon_f \delta_k^2.
\]}

\rev{Now fix a realization $\{\omega_k\}$ for the sequence $\{X_k,G_k,F^k_0,F^k_s\}$ in $\mathcal{F}_{\infty}$ and consider an iterate $x_k$ in this sequence.} We consider the different cases of an approximate stationarity condition denoted as:
\[
\|\rev{[}\nabla f(x_k)\rev{]}_{I_k}\| \le \epsilon,
\]

\paragraph{Case 1} $\|\rev{[}\nabla f(x_k)\rev{]}_{I_k}\| \geq \zeta \delta_k$.

We examine the following subcases based on different events:

\begin{enumerate}
    \item[(a)] \rev{$A_k\cap B_k$}: The model $g_k$ satisfies the $\kappa$-$\delta_k$ accuracy condition as well as having $\varepsilon_f$ accurate function evaluations. Applying~\eqref{eq:zetadef},
    \[
   \|\rev{[}\nabla f(x_k)\rev{]}_{I_k}\| \geq \delta_k /a. 
    \]
Rearranging, we obtain

\[
\delta_k \le a\|\rev{[}\nabla f(x_k)\rev{]}_{I_k}\| 
 \le \frac{a\max\left\{\delta_k,\alpha\|\rev{[}\nabla f(x_k)\rev{]}_{I_k}\| \right\}}{ \alpha } .
\]
Notice that this implies~\eqref{cond_delta_b}, that is,
\[
\delta_k \le a\Biggl\|x_k-P_{I_k}\biggl(x_k-\alpha\min\left\{1,\frac{\delta_k}{\alpha\|\nabla f(x_k)\|}\right\}\nabla f(x_k)\biggr)\Biggr\|,
\]
and so we can apply Lemma~\ref{lem:gooddec2} to conclude that
        \[
        f(x_k) - f(\hat{x}_k) \geq \frac{1}{2\alpha} \|\hat{x}_k - x_k\|^2.
        \]

Moreover, due to model accuracy it holds that
\[
\|g_k\|\ge \|\nabla f(x_k)\|-\kappa_g \delta_k\ge (\zeta-\kappa_g)\delta_k\ge 
\min\Biggl\{\frac 1{\eta_2}, \frac{1-\eta_1}{2 \varepsilon_f+\kappa_{\rev{f}} \delta_{max}}\Biggr\}\delta_k.
\]
As such, we can apply Lemma~\ref{lem:stepaccept} to conclude that the step is accepted and Lemma~\ref{lem:gooddec2} to conclude that the stochastic process proceeds as
\begin{equation}
\begin{split}
    \Phi_{k+1}-\Phi_k & \le -\nu c \delta_k \Biggl\|x_k-P_{I_k}\biggl(x_k-\alpha\min\left\{1,\frac{\delta_k}{\alpha\|\nabla f(x_k)\|}\right\}\nabla f(x_k)\biggr)\Biggr\|+(1-\nu)(\gamma^2-1)\delta_k^2 \\ & \le \left[-\nu c\zeta+(1-\nu)(\gamma^2-1)\right]\delta_k^2<0,
    \end{split}
\end{equation}
where the second inequality uses the case assumption.
    
    \item[(b)] \rev{$A_k\cap B^c_k$}: The function values $f^0_k, f^s_k$ do not satisfy the $\varepsilon_f$-accuracy condition, while model accuracy still holds. In this case the same argument as part (a) holds, with the caveat that erroneous function estimates could lead to a step rejection. In that case, the change in the stochastic process is bounded by~\eqref{eq:stochprocnotacc}, that is,
    \[
    \Phi_{k+1}-\Phi_k = (1-\nu)\left(\frac{1}{\gamma^2}-1\right)\delta_k^2 < 0.
    \]

    \item[(c)] \rev{$A_k^c\cap B_k$}: If the step is unsuccessful then again we can apply~\eqref{eq:stochprocnotacc}. Otherwise, with accurate function estimates, we know from Lemma~\ref{lem:stepaccept} together with~\eqref{eq:thmvarf} that in this case
    \[
    \Phi_{k+1}-\Phi_k\le \left[-\nu \eta_1\eta_2+(1-\nu)(\gamma^2-1)\right]\delta_k^2,
    \]
    which is still bounded by~\eqref{eq:stochprocnotacc} on account of~\eqref{eq:nudef}.
    
    \item[(d)] \rev{$A_k^c\cap B_k^c$} In this case, standard Lipschitz arguments give the following bound on the increase in the value of $\Phi$:
    \[
    \Phi_{k+1}-\Phi_k \le \nu C_L\|\rev{\left[\nabla f(x_k)\right ]_{I_k}}\|\delta_k+(1-\nu)(\gamma^2-1)\delta_k^2, \,\, C_L:=\left(1+\frac{3L}{2\zeta}\right) .
    \]
\end{enumerate}

We can finally combine these results to obtain, using the definitions of the probabilities $\theta$ and $\beta$,
\begin{align*}
\mathbb{E}\left[\Phi_{k+1}-\Phi_k\vert \mathcal{F}_k\right] \le & \theta\beta[-\nu c\|\|\rev{[}\nabla f(x_k)\rev{]}_{I_k}\|]\delta_k+(1-\nu)(\gamma^2-1)\delta_k]\\
& + [\theta(1-\beta)+(1-\theta)\beta](1-\nu)\left(\frac{1}{\gamma^2}-1\right)\delta_k^2 \\
& + (1-\theta)(1-\beta)\left[C_L\|\rev{\left[\nabla f(x_k)\right]_{I_k}}\|\delta_k+(1-\nu)(\gamma^2-1)\delta_k^2\right].
\end{align*}
We can observe that we can proceed along the same lines as the proof of Case 1 in~\cite[Theorem 4.11]{chen2018stochastic} to conclude that with $\theta,\beta$ chosen to satisfy
\begin{equation}\label{eq:thbcond}
\frac{(\theta\beta-1/2)}{(1-\theta)(1-\beta)} \ge \frac{C_L}{c},
\end{equation}
we can apply~\eqref{eq:nudef} to obtain that \rev{both of the following two conditions hold}:
\begin{equation}\label{eq:case1res1}
        \mathbb{E}\left[\Phi_{k+1}-\Phi_k\vert \mathcal{F}_k,\{\|\rev{\left[\nabla f(x_k)\right]_{I_k}}\|\ge \zeta \delta_k\}\right] \le -\frac{1}{4}c \nu \|\nabla f(x_k)\|\delta_k
\end{equation}
and
\begin{equation}\label{eq:case1res2}
    \mathbb{E}\left[\Phi_{k+1}-\Phi_k\vert \mathcal{F}_k,\{\|\rev{\left[\nabla f(x_k)\right]_{I_k}}\|\ge \zeta \delta_k\}\right] \le -\frac{1}{2}(1-\nu)(\gamma^2-1) \delta^2_k.
\end{equation}

\paragraph{Case 2: } $\|\rev{[}\nabla f(x_k)\rev{]}_{I_k}\| <\zeta \delta_k$.

If $\|g_k\|<\eta\delta_k$ then~\eqref{eq:stochprocnotacc} holds. Now assume that $\|g_k\|\ge \eta_2 \delta_k$. We again examine the following subcases based on different events:

\begin{enumerate}
    \item[(a)] \rev{$A_k\cap B_k$}: The model $g_k$ satisfies the $\kappa$-$\delta_k$ accuracy condition as well as having $\varepsilon_f$ accurate function evaluations. In this case, since it cannot be ensured that the step is accepted, we can apply the argument of Case 1c to conclude that again~\eqref{eq:stochprocnotacc} holds.

       \item[(b)] \rev{$A_k\cap B^c_k$}: The function values $f^0_k, f^s_k$ do not satisfy the $\varepsilon_f$-accuracy condition, while model accuracy still holds. An unsuccessful iteration yields~\eqref{eq:stochprocnotacc} a successful iteration satisfies
       \[
       f(x_k)-f(x_{k+1}) = f(x_k)-h_k(x_k)+h_k(x_k)-h_k(\hat{x}_k)+h_k(\hat{x}_k)-f(\hat{x}_k) \ge (\eta_2/\alpha -2\kappa_f)\delta_k^2 \ge \kappa_f \delta_k^2
       \]
with~\eqref{eq:thmetaacc} responsible for the last inequality. Finally~\eqref{eq:nudef} implies~\eqref{eq:stochprocnotacc} holds again. 

       \item[(c)] \rev{$A^c_k\cap B_k$}: It is the same as Case 1c.

              \item[(d)] \rev{$A^c_k\cap B^c_k$}: It is the same as Case 1d. 

\end{enumerate}
Now, with $\theta,\beta$ chosen such that
\begin{equation}\label{eq:thbetacond2}    (1-\theta)(1-\beta)\le \frac{(\gamma^2-1)(1-\nu)}{(1-\nu)(\gamma^4-1)+2\gamma^2 C_L \zeta\nu},
\end{equation}
we follow similar arguments to obtain
\begin{equation}\label{eq:case1res2}
    \mathbb{E}\left[\Phi_{k+1}-\Phi_k\vert \mathcal{F}_k,\{\|\rev{\left[\nabla f(x_k)\right]_{I_k}}\|<\zeta \delta_k\} \right] \le -\frac{1}{2}(1-\nu)\left(1-\frac{1}{\gamma^2}\right) \delta^2_k.
\end{equation}

Finally, combining the two cases yields that
\[
\mathbb{E}\left[\Phi_{k+1}-\Phi_k\vert \mathcal{F}_k\right] \le - \sigma \delta_k^2
\]
with $\sigma>0$, and \rev{with the application of standard Martingale Convergence Theory} the theorem has been proven. 
\end{proof}

\rev{By taking a look at the proof of the previous theorem, it is easy to see that the probabilities involved in the assumptions, i.e., $\theta$ and $\beta$, 
should be suitably chosen. This fact is reported in the following remark.
\begin{remark}
Following the same reasoning as in \cite{chen2018stochastic} Theorem 4.11 and Corollary 4.12, we need to suitably choose $\theta$ and $\beta$ so that both equations \eqref{eq:thbcond} and \eqref{eq:thbetacond2} are satisfied. By setting those parameters sufficiently large, as suggested in \cite{chen2018stochastic}, we have that those conditions are satisfied. 
\end{remark}
}
We may proceed now to the main and final result. The rest of the original convergence argument can be applied directly to $\|\rev{[}\nabla f(x_k)\rev{]}_{I_k}\|$. However, recall that this is not the object that is of primary interest. We are indeed interested in proving that the proposed algorithm gives us a point satisfying some suitable optimality condition with high probability\rev{, specifically $L$-stationarity expressed in Definition~\ref{def:ls}}.

\begin{theorem} \rev{Let all the assumptions of Theorem \ref{thm:4.11_equivalent} hold, with $\theta$ and $\beta$ sufficiently large.  Then   \begin{equation}\label{eq:limconv}
    \lim\limits_{k\to\infty}\|\rev{[\nabla f(x_k)]_{I_k}}\| = 0
    \end{equation}
almost surely.    
}
    Moreover, for any limit point $x^*$ of a realization of iterates $\{x_k\}$,
    \rev{we have that there exists some $\bar L>0$ such that $x^*$ satisfies $\bar L$-stationarity, as given by  Definition~\ref{def:ls}.} 
\end{theorem}

\begin{proof}
\rev{The first part of the statement follows directly from the identical arguments as in~\cite{chen2018stochastic}. Specifically:}
\begin{enumerate}
    \item \rev{From Theorem~\ref{thm:4.11_equivalent}, we have that $\sum \delta_k^2<\infty$ holds almost surely, and thus $\delta_k\to 0$ almost surely. Then, one can obtain a contradiction to there being some $\epsilon'$ for which $\|[\nabla f(x_k)]_{I_k}\|\ge \epsilon'$, implying $\liminf_{k \to \infty} \|[\nabla f(x_k)]_{I_k}\| = 0$ as in~\cite[Theorem 4.16]{chen2018stochastic}.}
    \item \rev{As in~\cite[Lemma 4.17]{chen2018stochastic}, if $K_{\epsilon}$ is a subsequence of iterations such that $\|[\nabla f(x_k)]_{I_k}\|>\epsilon$ then \mbox{$\sum\limits_{k\in K_{\epsilon}} \delta_k < \infty$} by similarly invoking the condition that $\|[\nabla f(x_k)]_{I_k}\|\ge \zeta\delta_k$ and proving that, almost surely,
    \begin{equation}\label{eq:deltasumfin}
          \sum\limits_{k\in K_{\epsilon}} \Delta_k<\infty.  
        \end{equation}}
        \item \rev{Finally, using the previous results, applying the arguments of~\cite[Theorem 4.18]{chen2018stochastic} one can show that, if $\limsup\|[\nabla f(x_k)]_{I_k}\|> \epsilon$, then $\sum\limits_{K_{\epsilon}} \delta_k=\infty$, obtaining a contradiction with the previous result. This establishes~\eqref{eq:limconv}.}
    \end{enumerate}

\rev{Consider now an almost sure realization and let $x^*$ be a limit point of $\{x_k\}$, that is, there exists a subsequence $\{x_k\}_S\to x^*$, with $S \subseteq \{0,1,\ldots\}$.} Now fix the realization \rev{and corresponding subsequence} for the remainder of the proof.
\rev{Since $\sigma_k$ and $I_k$ are subsets of the finite set $\{1,\ldots,n\}$ for any $k$, without loss of generality we can assume that they are constant over the considered subsequence $S$ (passing into a further subsequence if needed). Namely, after discarding an appropriate finite sequence from the beginning of $S$, 
\begin{align}
I_k = I^* \quad \forall k \in S, \label{Ik_constant} \\
\sigma_k = \sigma^* \quad \forall k \in S. \label{sigmak_constant}
\end{align}}   
\rev{Observe that, for any $k$, we can write
\[
\alpha\min\left\{1,\frac{\delta_k}{\alpha\left\|g_k\right\|}\right\}\|g_k\|
\begin{cases}
= 0 \quad & \text{if $\|g_k\| = 0$}, \\
\le \dfrac{\alpha \delta_k\|g_k\|}{\alpha \|g_k\|} = \delta_k \quad & \text{if $\|g_k\| > 0$}.
\end{cases}
\]
Taking into account that $\{\delta_k\} \to 0$ from Theorem~\ref{thm:4.11_equivalent}, we get
\[
\lim_{k \to \infty} \alpha\min\left\{1,\frac{\delta_k}{\alpha\left\|g_k\right\|}\right\}\|g_k\| = 0.
\]
Thus,
\begin{equation}\label{lim_xk}
\lim_{k \to \infty, \, k \in S} x_k -\alpha\min\left\{1,\frac{\delta_k}{\alpha\left\|g_k\right\|}\right\}\|g_k\| = x^*.
\end{equation}
Since, according to~\eqref{sigmak_constant} and the definition of $\sigma_k$ given in~\eqref{eq:sigk}, we have
\[
\sigma^* \in \tilde{\Sigma}\left(x_k-\alpha\min\left\{1,\frac{\delta_k}{\alpha\left\|g_k\right\|}\right\}g_k\right) \quad \forall k \in S,
\]
it follows from~\eqref{lim_xk} that
\[
\sigma^* \in \tilde{\Sigma}(x^*).
\]
Hence, recalling~\eqref{Ik_constant} and the definition of $I_k$ given in~\eqref{eq:ikdef}, it holds that
\begin{equation}\label{inactive}
i \in  I_{\mathcal{I}}(x^*) \quad \Rightarrow \quad i \in \mathcal{I}^*,
\end{equation}
where we have used the fact that $x^*$ is feasible (and then, $x^*_i \ne 0$ implies that the $i$th component is one of the $K$ largest ones in $x^*$). Using~\eqref{inactive} and~\eqref{eq:limconv}, it follows that
\[
[\nabla f(x^*)]_i = 0 \quad \forall i \in  I_{\mathcal{I}}(x^*).
\]
We have thus proven that $x^*$ satisfies Basic Feasibility, according to Definition~\ref{def:bf}.
}

\rev{We now consider two different cases: 
\begin{itemize}
\item $\nabla f(x^*)= 0$. In this case it is easy to see that the point is $\bar L$-stationary for any choice of $\bar L>0$.
\item $\nabla f(x^*)\neq 0$. Recalling Lemma~\ref{lem:lstatb} and reasoning as in~\cite[Remark~2.3]{beck2013sparsity}, we have that $\bar L$-stationarity holds with
\[
\bar L = \max_{i \in I_{\mathcal{A}}(x^*)} \frac{|[\nabla f(x^*)]_i|}{M_K(x^*)}.
\]
\end{itemize}
}
\end{proof}

\section{Numerical Results}\label{sec:numerical}
In this section, we present two machine learning applications of Algorithm~\ref{alg:iht}: adversarial attacks on neural networks and the reconstruction of sparse Gaussian graphical models. The implementation was carried out using the Python programming language, using the \texttt{NumPy}, \texttt{Keras}, \texttt{Tensorflow}, \texttt{scikit-learn}, and \texttt{Pandas} libraries. The hyperparameters were selected as follows: $\eta_1 = 10^{-4}$, $\eta_2 = 10^{-4}$, $\delta_0 = 1$, $\delta_{\text{max}} = 10$, and $\gamma = 2$. 
All the experiments were conducted on a machine equipped with an 11th Gen Intel(R) Core(TM) i7-1165G7 CPU @ 2.80GHz (1.69 GHz). The code is available at \url{https://github.com/Berga53/Probabilistic_iterative_hard_thresholding}.


\subsection{Adversarial Attacks on Neural Networks}
Adversarial attacks are techniques used to craft imperceptible perturbations that, when added to regular data inputs, induce misclassifications in neural network models. These perturbations are typically designed to evade human detection while successfully fooling the model's classification process. 
One of the most powerful type of adversarial attack is the Carlini and Wagner \cite{DBLP:journals/corr/CarliniW16a}, characterized by the following formulation:
\rev{
\begin{equation}\label{eq:C&W}
\begin{split}
\min_{\delta} D(x, x + \Delta) + c \cdot f(x + \Delta) \\
\text{ such that } x + \Delta \in [0, 1]^n\ ,
\end{split}
\end{equation}
with $\Delta$ being the perturbation, $D$ being usually the $\ell_2$ or $\ell_0$ distance, and 
\begin{equation}
\rev{f(x) = \left( \max_{i \neq t} ([F(x)]_i) - [F(x)]_t \right)^+,}
\label{eq:fcw}
\end{equation}
where $[F(x)]_i$ is the probability output for the class $i$, and $t$ is the targeted class.}

Using our algorithm, we can incorporate the $\ell_0$ penalty directly in the constraint, so our final formulation of the problem is

\begin{equation}\label{eq:C&Wfinal}
\begin{split}
\min_{\|\delta\|_0 \leq K} \|\Delta\|_2 + c \cdot f(x + \Delta) \\
\text{ such that } x + \Delta \in [0, 1]^n\ .
\end{split}
\end{equation}

In practice, this allows us to decide how many pixels to perturb during the attack. While usual attacks are trained against selected samples of the dataset, in this paper, we will demonstrate a universal adversarial attack: the attack is performed against the entirety of the dataset, producing only one global perturbation. We will show that, in both targeted and untargeted attacks, we can significantly lower a model's accuracy using very few pixels.
We tested the attack on the MNIST dataset, which consists of 60,000 images of handwritten digits (0-9) that are 28 $\times$ 28 pixels in size. We performed both targeted and untargeted attacks. In the targeted attack, we aimed to misclassify the images into a specific class, \rev{using a different digit as target class in different experiments. This approach allows us to manipulate the model to misclassify any digit as any chosen target digit.} In the untargeted attack, we simply aimed to cause any misclassification, \rev{choosing, for every sample, the easiest class to target}. However, the untargeted attack is generally a bit weaker in the context of the Carlini and Wagner Attack. We will show that, in both targeted and untargeted attacks, we can significantly lower a model's accuracy using very few pixels. We gradually increase the sparsity constraint and observe that this gradually increases the errors made by the model. In particular, in Figure \ref{fig:results_k}, we can see both the accuracy decreasing and the number of samples predicted as the attack target increasing, indicating that the attack is performed as desired. \rev{We present the mean and variance of the different experiments in blue, the best-targeted attack in orange, and the untargeted attack in green. As expected, the untargeted attack is weaker than the targeted one. In Figure~\ref{fig:advexample}, we observe examples of both the original and perturbed images where the attacks were successful, specifically targeting the digit 5.} 

\begin{figure}[h]
    \centering
    \begin{minipage}[b]{0.45\textwidth}
        \centering
        \includegraphics[width=\textwidth]{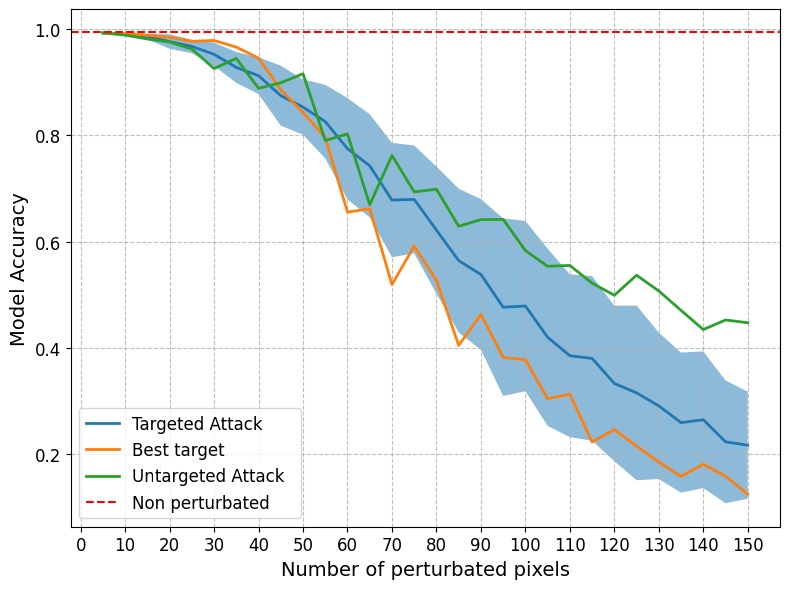}
        \label{fig:accuracy_k}
    \end{minipage}
    \begin{minipage}[b]{0.45\textwidth}
        \centering
        \includegraphics[width=\textwidth]{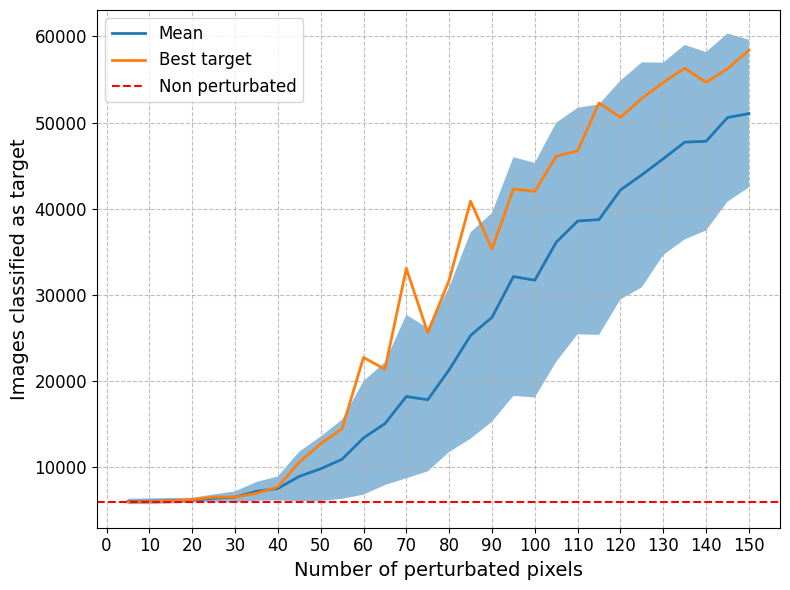}
        \label{fig:target_k}
    \end{minipage}
    \caption{Effect of increasing the sparsity constraint on accuracy and targeted attack predictions.}
    \label{fig:results_k}
\end{figure}

\begin{figure}[h]
    \centering
    \includegraphics[width=\linewidth]{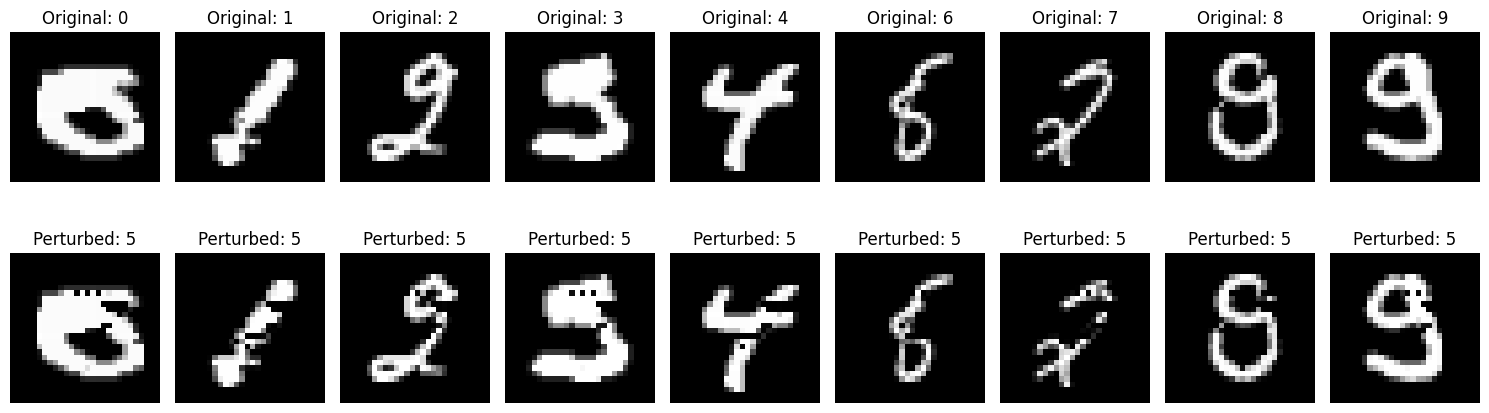}
    \caption{Example of perturbed images with $\| \delta\|_0 = 25$ and target 5}
    \label{fig:advexample}
\end{figure}

\subsection{Sparse Gaussian Graphical Models}
Probabilistic Graphical Models are a popular tool in machine learning to model the relationships between random variables. The Gaussian Graphical Model is an undirected graph with each edge corresponding to a Gaussian conditional probability of one variable at the end of the edge to another. By learning the adjacency matrix together with the model weights, we can infer the proximal physical and possibly causal relationships between quantities. 

This is of special importance in high dimensional settings (see, e.g.,~\cite{wainwright2019high}). Whereas in many contemporary ``big data'' approaches the sample size is many orders of magnitudes larger than the dimensionality of feature space, there are a number of settings wherein obtaining data samples is costly, and such a regime cannot be expected to hold. Indeed this is often the case in medical applications, wherein recruiting volunteers for a clinical trial, or even obtaining health records, presents formidable costs to significant scaling in sample size. On the other hand, the precision of instrumentation has led to detailed personal physiological and biomarker data, yielding a very high dimensional feature space. One associated observation is that in the underdetermined case, when the dimensionality of the features exceeds the number of samples, some of the guarantees associated with the $\ell_1$ proxy for sparsity are no longer applicable, bringing greater practical salience to having a reliable algorithm enforcing sparsity explicitly.

The recent work \cite{behdin2023sparse} presented an integer programming formulation for training sparse Gaussian graphical models. Prior to redefining the sparsity regularization using binary variables, their $\ell_0$ optimization problem is given as
\rev{
\begin{equation}\label{eq:sparsegausspgm}
    \min\limits_{W\in\mathbb{S}^p} \, F_0(W) := \sum\limits_{i=1}^p \left(-\log(w_{ii})+\frac{1}{w_{ii}}\|\tilde{X} w_i\|^2\right)+\lambda_0 \|W\|_0+\lambda_2\|W\|_2^2\ ,
\end{equation}
with $\tilde{X}=\frac{1}{\sqrt{n}}X$ the scaled feature matrix and $X\in\mathbb{R}^{p\times n}$ a matrix consisting of $p$ measures and $n$ samples, $W\in\mathbb{S}^p$ the weight matrix related to the graph, $\mathbb{S}^p$ being the set of symmetric matrices in $\mathbb{R}^{p\times p}$, $w_{ii}$ indicating the $i$-th component in the diagonal of $W$ and $w_i$ indicating the $i$-th row of $W$.} Functionally, $w_{ij}$ defines an edge between node $i$ and $j$ in the graph, with a nonzero indicating the presence of an active edge, which corresponds to a direct link in the perspective of DAG structure of the group. The value associated with the edge corresponds to the weight defining the strength of the interaction between the features $i$ and $j$. We seek to regularize cardinality for the sake of encouraging parsimonious models, as well as minimizing the total norm of the weights for general regularization.

Due to the structure of our algorithm, we can modify the formulation of the problem by incorporating the $\ell_0$ constraint. The final formulation of the problem is then expressed as follows:

\begin{equation}\label{eq:sparsegausspgmfinal}
    \min\limits_{W\in\mathbb{S}^p, \|W\|_0 \leq K} \, F_0(W) := \sum\limits_{i=1}^p \left(-\log(w_{ii})+\frac{1}{w_{ii}}\|\tilde{X} w_i\|^2\right) +\lambda_2\|W\|_2^2\ .
\end{equation}

We also observed that the $\ell_0$ constraint in our formulation is very strong. In practical applications, we eliminate $\lambda_2$ penalty term, as the $\ell_0$ constraint was the dominant factor in the model.

We applied the model to the GDS2910 dataset from the Gene Expression Omnibus (GEO). This dataset consists of gene expression profiles, which naturally yield a high-dimensional feature space, with 1900 features and 191 samples. Given this feature-to-sample ratio, we can assume some level of sparsity in the final adjacency matrix. Since there is no ground truth for the underlying structure, our goal is to investigate how changing the $\ell_0$ constraint affects the results of our method, while also gathering information on the true sparsity nature of the data. We performed the test by gradually increasing $K$, the $\ell_0$ constraint, from 5000 to 15000. This range was previously determined to be optimal based on preliminary tests. Note that the adjacency matrix we are searching for is of size $1900 \times 1900$, resulting in a total of $3.61 \cdot 10^6$ entries. To ensure the robustness of the results, for each value of $K$, we performed ten runs starting from different randomly chosen feasible points, and the algorithm was given a total of 1000 iteration for every run. We also decided to set the $\lambda_2$ parameter to zero, as we observed that the strong $\ell_0$ constraint was dominant over the $\ell_2$ penalty.

We also divided the dataset into training and validation sets to determine whether the reconstructed matrix is a result of overfitting. In Figure \ref{fig:effect_of_K}, we show the effect of varying $K$, which represents the number of nonzero entries that the matrix is allowed to have. The figure on the left, which shows the average objective value found over the ten runs, demonstrates that increasing $K$ eventually stops being beneficial to the model's performance. Additionally, we observe that the number of mean accepted iterations also stops increasing, indicating that the model cannot extract more information from the data. This suggests that the true sparsity of the data can be estimated by identifying the point at which further increasing $K$ no longer improves the model's results. In Figure \ref{fig:fovert}, we present an example from our tests where the objective function decreases over the successful iterations.

\begin{figure}[h]
    \centering
    \begin{minipage}[b]{0.45\textwidth}
        \centering
        \includegraphics[width=\textwidth]{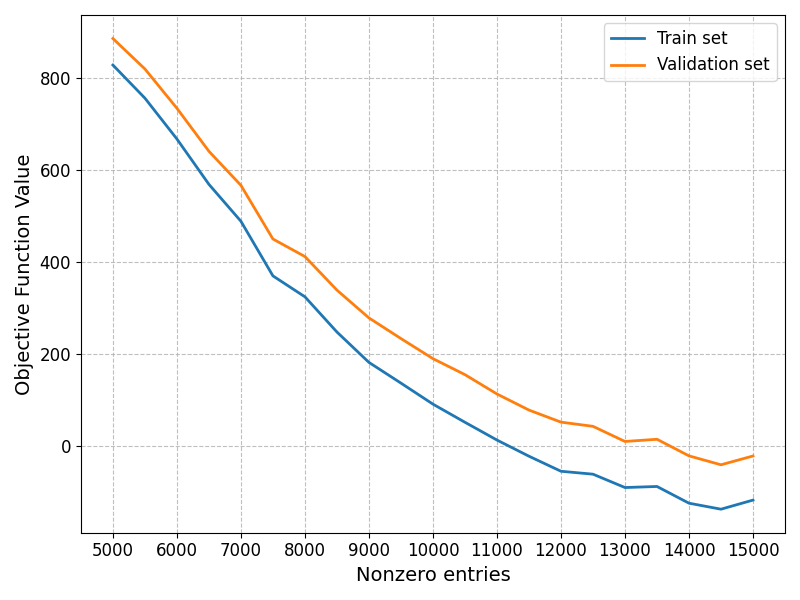}
        \label{fig:accuracy_k}
    \end{minipage}
    \begin{minipage}[b]{0.45\textwidth}
        \centering
        \includegraphics[width=\textwidth]{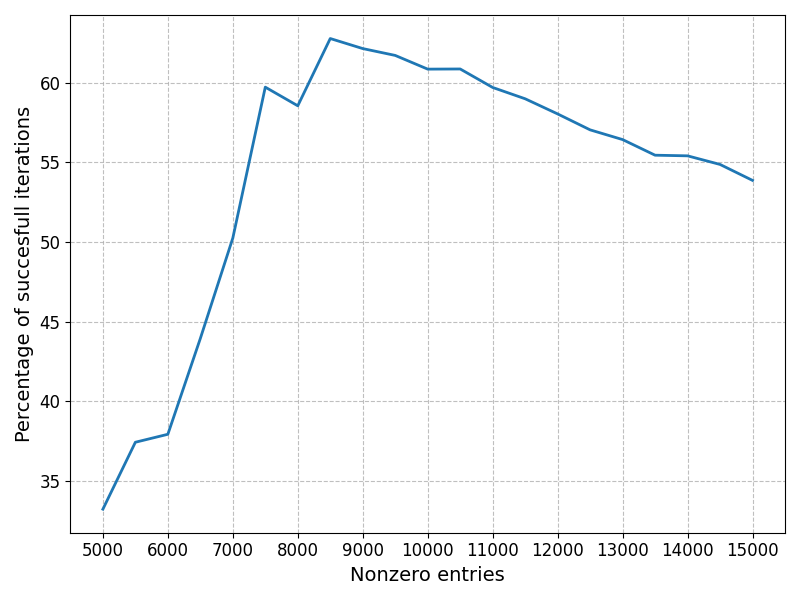}
        \label{fig:target_k}
    \end{minipage}
    \caption{Effect of increasing the sparsity constraint $K$.}
    \label{fig:effect_of_K}
\end{figure}

\begin{figure}[h]
    \centering
    \includegraphics[width=0.5\linewidth]{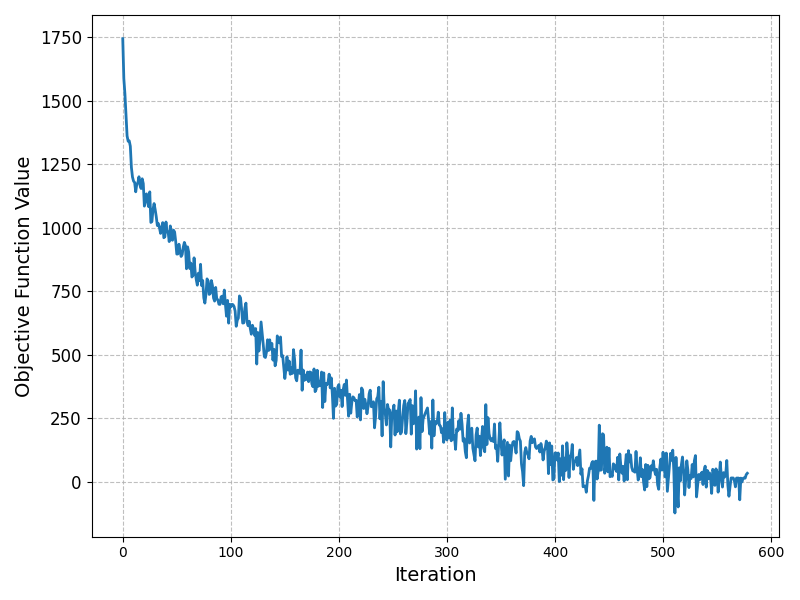}
    \caption{Objective function over the iterations.}
    \label{fig:fovert}
\end{figure}

\section{Conclusions}\label{sec:conclusions}

In this paper, we addressed the stochastic cardinality-constrained optimization problem, providing a well defined algorithm, convergence theory and illustrative experiments. Many contemporary machine learning applications involve scenarios where sparsity is crucial for high-dimensional model fitting. We proposed an iterative hard-thresholding like algorithm based on probabilistic models that nicely  balances computational efficiency and solution precision by allowing flexible gradient estimates while incorporating hard sparsity constraints.

We analyzed the theoretical properties of the method and proved almost sure convergence to L-stationary points under mild assumptions. This extends previous work in the optimization literature on finding solutions with strong stationarity guarantees together with machine learning articles that perform iterative hard thresholding with stochastic gradients to achieve a novel balance between ease of a fast implementation and formal guarantees of performance. The numerical experiments confirmed the practical effectiveness of our method, showcasing its potential in machine learning tasks such as adversarial attacks and probabilistic graphical model training. By enforcing explicit cardinality constraints, our approach was able to produce models with enhanced sparsity and interpretability in the end.

Future work may involve extending the algorithm to accommodate additional nonlinear  constraints,  exploring techniques to further improve scalability and performance, as well as
testing the algorithm on some other relevant Machine Learning applications, like, e.g., sparse  Dynamic Bayesian Network training.

\par\bigskip\noindent
\textbf{Funding} The work of Vyacheslav Kungurtsev was funded by the European Union’s Horizon Europe research and innovation programme under grant agreement No. 101084642.

\bibliography{sample}

\end{document}